\documentclass{amsart}
\usepackage{delimset, enumitem}
\setlist[enumerate,itemize]{itemsep=3pt}

\usepackage[colorlinks]{hyperref}
\usepackage[capitalize]{cleveref}

\newtheorem{theorem}{Theorem}[section]

\newtheorem{lemma}[theorem]{Lemma}
\newtheorem{conjecture}[theorem]{Conjecture}
\newtheorem{proposition}[theorem]{Proposition}

\newcommand{\A}{\mathcal{A}}
\usepackage{mathrsfs} % \mathscr{*}
\newcommand{\B}{\mathscr{B}}
\renewcommand{\P}{\mathscr{P}}

\newcommand\floor[1]{\lfloor#1\rfloor}
\newcommand{\ovD}[1]{\overline{\mbox{$#1$}\raisebox{0mm}{}}}
\newcommand{\ovU}[1]{\overline{\mbox{$#1$}\raisebox{2.5mm}{}}}

\DeclareMathOperator{\des}{\mathrm{des}}
\DeclareMathOperator{\asc}{\mathrm{asc}}
\DeclareMathOperator{\cdes}{\mathrm{cdes}}
\DeclareMathOperator{\casc}{\mathrm{casc}}

\crefname{eqrl}{Equivalence}{Equivalences}
\creflabelformat{eqrl}{~\upshape(#2#1#3)}
\crefname{itm}{}{}
\creflabelformat{itm}{~\upshape#2#1#3}

\numberwithin{equation}{section}
\numberwithin{figure}{section}
\numberwithin{table}{section}

\begin{document}

\title[A Toeplitz property on certain permutations]
{A Toeplitz property of ballot permutations\\ and odd order permutations}

\author[D.G.L. Wang]{David G.L. Wang$^\dag$$^\ddag$}
\address{
$^\dag$School of Mathematics and Statistics, Beijing Institute of 
Technology, 102488 Beijing, P. R. China\\
$^\ddag$Beijing Key Laboratory on MCAACI, Beijing Institute of 
Technology, 102488 Beijing, P. R. China}
\email{glw@bit.edu.cn; david.combin@gmail.com}

\author[J.J.R. Zhang]{Jerry J.R. Zhang}
\address{School of Mathematics and Statistics, Beijing Institute of 
Technology, 102488 Beijing, P. R. China}
\email{jrzhang.combin@gmail.com}


%05A05 Permutations, words, matrices
%05A15 Exact enumeration problems, generating functions 
%05A19 Combinatorial identities, bijective combinatorics
%15B05 Toeplitz, Cauchy, and related matrices

\thanks{Corresponding author: David G.L.\ Wang.
This paper was supported by General Program of National Natural 
Science Foundation of China (Grant No.\ 11671037).}

\begin{abstract}   
We give a new semi-combinatorial proof for the equality of the number of ballot permutations of length $n$ and the number of odd order permutations of length $n$, which is due to Bernardi, Duplantier and Nadeau. Spiro conjectures that the descent number of ballot permutations and certain cyclic weight of odd order permutations of the same length are equi-distributed. We present a bijection to establish a Toeplitz property for ballot permutations with any fixed number of descents, and a Toeplitz property for odd order permutations with any fixed cyclic weight. This allows us to refine Spiro's conjecture by tracking the neighbors of the largest letter in permutations. 
\end{abstract}

\maketitle
%\tableofcontents

\section{Introduction}
Let $\mathfrak{S}_n$ be the symmetric group
of permutations of the set $[n]=\{1,2,\dots,n\}$.
Let $\pi=\pi_1\pi_2\dotsm\pi_n\in\mathfrak{S}_n$.
The \emph{signature} of $\pi$ is defined to be
the sequence $(q_1,q_2,\dots,q_{n-1})$
where 
\[
q_i=\begin{cases}
-1,&\text{if $\pi_i>\pi_{i+1}$};\\
1,&\text{if $\pi_i<\pi_{i+1}$}.
\end{cases}
\]
Niven~\cite{Niv68} found a determinantal formula for the number 
of permutations of length~$n$ with a prescribed signature,
and showed that this number attains its maximum if and only if
the signature is for an Andr\'e permutation~\cite{And1881};
see de Bruijn~\cite{deB70} for a recursive proof.
A pair $(\pi_i,\,\pi_{i+1})$ of letters
is a \emph{descent} (resp., an \emph{ascent}) 
if $\pi_i<\pi_{i+1}$ (resp., $\pi_i>\pi_{i+1}$).
Denote the number of descents (resp., ascents) of $\pi$ by $\des(\pi)$
(resp., $\asc(\pi)$). 
We call the number
\[
h(\pi)
=\asc(\pi_1\pi_2\dotsm\pi_n)-\des(\pi_1\pi_2\dotsm\pi_n)
\]
the \emph{height} of $\pi$.
The permutation $\pi$ is said to be a \emph{ballot permutation} 
if the height of any prefix of~$\pi$ is nonnegative, namely,
$h(\pi_1\pi_2\cdots\pi_i)\ge 0$
for all $i\in[n]$.
The number of ballot permutations of height $0$ in $\mathfrak{S}_{2n+1}$, 
or \emph{Dyck permutations} of length~$n$,
is the Eulerian-Catalan number; see Bidkhori and Sullivant~\cite{BS11}.
A classical coin-tossing game problem
concerning the descent-ascent structure 
in a sequence of independent random variables of values $\pm1$
was considered by Chung and Feller~\cite{CF49}.

In this paper, we concern the following beautiful result
which is due to Bernardi, Duplantier and Nadeau \cite{BDN10}.

\begin{theorem}[Bernardi et al.]\label{thm:b=p}
The number of ballot permutations of length $n$ is
\[
p_n=\begin{cases}
(n-1)!!^2,&\text{if $n$ is even},\\[3pt]
n!!\cdot (n-2)!!,&\text{if $n$ is odd},
\end{cases}
\]
where $(2m-1)!!=(2m-1)(2m-3)\cdots3\cdot 1$.
\end{theorem}

Bernardi et al.~\cite{BDN10} obtained \cref{thm:b=p}
by considering more generalized paths containing horizontal steps, 
called \emph{well-labelled positive paths}. 
They constructed a bijection between well-labelled positive paths 
of size $n$ with $k$ horizontal steps and matchings on $[2n]$
having $k$ pairs $(i,j)$ with $i\in[n]$ and $j\in[n+1,\,\dotsc,\,2n-1]$,
and thus obtained an explicit formula 
for the number of well-labelled positive paths of size $n$ 
having $k$ horizontal steps. Taking $k=0$ in their formula yields \cref{thm:b=p}.

Denote by $\P_n$ the set of odd order permutations 
of $[n]$, viz., the set of permutations of $[n]$ whose every cycle is of odd length. 
By considering the neighbours of the letter $n$, we see that
\[
\abs{\P_n}=\abs{\P_{n-1}}+(n-1)(n-2)\abs{\P_{n-2}}.
\]
This recurrence gives 
\[
\abs{\P_n}=p_n
\]
immediately, where $p_n$ is the number defined in \cref{thm:b=p}.
In order to find an analogue for the descent statistic in the context
of odd order permutations, Spiro~\cite{Spi18X} 
introduced the following interesting notion. 
For a permutation $\pi$, Spiro defines
\[
M(\pi)=\sum_{c}\min\brk1{\cdes(c),\,casc(c)},
\]
where the sum runs over all cycles of $\pi$, with the \emph{cyclic descent} 
\[
\cdes(c)=\abs{\{i\in[k]\colon c_i>c_{i+1}\ \text{where $c_{k+1}=c_1$}\}},
\]
and the \emph{cyclic ascent}
\[
\casc(c)=\abs{\{i\in[k]\colon c_i<c_{i+1}\ \text{where $c_{k+1}=c_1$}\}}
=\abs{c}-\cdes(c),
\]
where $\abs{c}$ is the length the cycle $c$.
We call 
\[
w(c)=\min\brk1{\cdes(c),\,casc(c)}
\]
the \emph{cyclic weight} of $c$, and $w(\pi)=M(\pi)$
the \emph{cyclic weight} of $\pi$.

\begin{conjecture}[Spiro]\label{conj:Spiro}
Let $n\ge 1$ and $0\le d\le \floor{(n-1)/2}$.
Then the number of ballot permutations of length $n$ 
with descent $d$
equals the number of odd order permutations of length $n$
with cyclic weight~$d$.
\end{conjecture}
Spiro confirmed \cref{conj:Spiro} for $d\le 3$,
$d=\floor{(n-1)/2}$, and $d=(n-3)/2$.

In the next section we give a new proof of \cref{thm:b=p} 
by using certain combinatorial decomposition of ballot permutations
with respect to the neighbors of the letter $n$.
In \cref{sec:Toeplitz},
we display a bijection to establish 
a Toeplitz property for the number of ballot permutations of length $n$
with a fixed number of descents.
A slight modification of the bijection gives the Toeplitz property
for the number of odd order permutations of length $n$ with a fixed cyclic weight.
These Toeplitz properties lead us to a refinement conjecture of \cref{conj:Spiro};
see \cref{conj:nd1j}. In \cref{sec:easy},
we show some easy cases of \cref{conj:nd1j}.

\section{A new semi-bijective proof of \cref{thm:b=p}}\label{sec:reproof}

We give an overview of notion and notation in combinatorics on words 
which will be of use; see \cite{Lot02B,Lot05B}. 
For any word $w$ of length~$n$,
we denote its $i$th letter by~$w_i$, 
denote its alphabet
$\A(\alpha)=\{w_1,\dots,w_n\}$,
denote its length~$n$ by~$\ell(w)$,
and write $w_{-1}=w_n$.
Denote the \emph{reversal} of $w$ by $w'=w_nw_{n-1}\dotsm w_1$.
In particular, the reversal of the empty word
$\epsilon$ is $\epsilon$ itself.
We call a word $u$ a \emph{factor} (resp., \emph{prefix}) of~$w$ 
if there exist words $x$ and~$y$ such that $w=xuy$ (resp., $w=uy$). 
We say that~$w$ is \emph{ballot} if 
$h(u)\ge0$ for any prefix $u$ of $w$.

The notions of reversal and factor have a cyclic version.
Let $\pi=(c_1)\cdots (c_k)\in\mathfrak{S}_n$, where $(c_i)$ are the cycles of $\pi$.
We say that a word $u$ 
is a \emph{cyclic factor} of~$\pi$ if $u$ is a factor of some word $v$ 
such that~$(v)$ is a cycle of $\pi$.
For example, the permutation $\pi=(145)(26837)$ has a cyclic factor $(372)$. 

Denote by $\B_n$ the set of ballot permutations of $[n]$. 
Denote
\[
\B_{n,d}=\{\pi\in\B_n\colon \des(\pi)=d\}
\quad\text{and}\quad
\P_{n,d}=\{\pi\in\P_n\colon w(\pi)=d\}.
\]
For any letters $i$ and~$j$,
denote by $\B_{n,d}(i,j)$ 
the set of permutations in $\B_{n,d}$ that contain the factor~$inj$, 
by $\P_{n,d}(i,j)$
the set of permutations in $\P_{n,d}$ that contain the cyclic factor $inj$.
For example, we have 
\begin{alignat*}{2}
\B_{4,1}(1,3)&=\emptyset,\quad\quad&
\B_{4,1}(3,1)&=\{2341,\,3412\},\\
\P_{4,1}(1,3)&=\{(143)\},\quad\quad&
\P_{4,1}(3,1)&=\{(341)\}.
\end{alignat*}
We use lower-case letters $b$ and $p$, 
in replace of $\B$ and $\P$ respectively, 
to denote the corresponding set cardinality, 
such as $p_{n,d}(i,j)=\abs{\P_{n,d}(i,j)}$.
It is clear that
\[
b_{n,d}=p_{n,d}=\begin{cases}
1,&\text{if $d=0$};\\[3pt]
0,&\text{if $d>\floor{(n-1)/2}$}.
\end{cases}
\]

\begin{lemma}\label{lem:|i-j|=1}
For any $i,j\in[n]$ such that $\abs{i-j}=1$, 
\[
b_{n,d}(i,j)=b_{n-2,\,d-1}
\quad\text{and}\quad
p_{n,d}(i,j)=p_{n-2,\,d-1}.
\]
\end{lemma}
\begin{proof}
Fix a letter $i$. Suppose that $i<j$. Then $j=i+1$.
For $\pi\in\B_{n,d}(i,\,i+1)$,
define $\phi(\pi)$ to be the permutation obtained from $\pi$ by
removing the letters $i+1$ and $n$
and arranging the remaining letters in the order-preserving manner
so that $\phi(\pi)\in\mathfrak{S}_{n-2}$. 
It is clear that $\phi$ is a bijection between the sets
$\B_{n,d}(i,\,i+1)$ and $\B_{n-2,\,d-1}$. 
The other case $i>j$ can be handled similarly.
For odd order permutations, one obtains the desired equality 
by using the same operation of $\phi$ and keeping the cycle structure.
\end{proof}

For any permutation $\pi$ on a set of $n$ positive integers, 
we define the \emph{standard form} of $\pi$ 
to be the permutation $\sigma\in\mathfrak{S}_n$
such that $\pi_i<\pi_j$ if and only if $\sigma_i<\sigma_j$
for all pairs $(i,j)$.
Denote by $\B$ the set of finite permutations on positive integers 
whose standard forms are ballot permutations.
Let $\omega\in\B\setminus\{\epsilon\}$. 
We say that a ballot permutation $\pi\in\B_n$
is \emph{$\omega$-decomposable} if 
$\pi$ is the concatenation $\alpha\omega\gamma\delta$ 
of factors $\alpha,\omega,\gamma,\delta$ such that 
\begin{equation}\label{cond:omega}
h(\alpha\omega\gamma)=h(\omega),
\end{equation}
where $\alpha$, $\gamma$ and $\delta$ are allowed to be the empty word $\epsilon$.
Let $X_n(\omega)$ be the set of 
$\omega$-decomposable ballot permutations $\pi\in\mathfrak{S}_n$. 
Define the \emph{$\omega$-decomposition} of a permutation $\pi\in X_n(\omega)$
to be the $4$-tuple $(\alpha,\omega,\gamma,\delta)$ 
such that $\pi=\alpha\omega\gamma\delta$ and that~$\gamma$ is 
the longest word satisfying
\[
\gamma'\omega_{-1}\in\B.
\]
From definition we see that the $\omega$-decomposition 
of any permutation in $X_n(\omega)$ uniquely exists.
In this case, we write 
$\pi=(\alpha,\omega,\gamma,\delta)$.
In addition,
we will use the convenience
$\gamma_{-1}= \omega_{-1}$
if $\gamma=\epsilon$.

\begin{lemma}\label{lem:delta}
Suppose that $\pi=(\alpha,\omega,\gamma,\delta)$.
If $\delta\not\in\B$, then $\gamma_{-1}>\delta_1$ and $h(\omega)\ne 1$.
\end{lemma}
\begin{proof}
Denote by $\chi$ the characteristic function defined by 
$\chi(P)=1$ if a proposition~$P$ is true,
and $\chi(P)=0$ if $P$ is false.
Suppose that $\delta\not\in\B$. Then $\ell(\delta)\ge 2$.
Let~$j$ be the minimum index such that 
\[
h(\delta_1\delta_2\cdots\delta_j)=-1.
\]
Then the factor $\eta=\gamma\delta_1\delta_2\cdots\delta_j$,
which is longer than~$\gamma$, satisfies 
\begin{align}
h(\alpha\omega\eta)
&=h(\alpha\omega\gamma)+(-1)^{\chi(\gamma_{-1}>\delta_1)}
+h(\delta_1\delta_2\cdots \delta_j)\notag\\
&=h(\omega)+(-1)^{\chi(\gamma_{-1}>\delta_1)}-1.\label{pf:omega}
\end{align}
On the other hand, note that
\[
h(\delta_{j-1}\cdots\delta_2\delta_1)=h(\delta_1\delta_2\cdots\delta_{j-1})=0
\quad\text{and}\quad
\delta_{j-1}\cdots\delta_2\delta_1\in\B.
\]
If $\gamma_{-1}<\delta_1$, then 
\[
h(\delta_j\delta_{j-1}\cdots\delta_1\gamma_{-1})=0
\quad\text{and}\quad
\delta_j\delta_{j-1}\cdots\delta_1\gamma_{-1}\in\B.
\]
Since $\gamma'\omega_{-1}\in\B$, we find $\eta'\omega_{-1}\in\B$. 
In this case, \cref{pf:omega} reduces to 
$h(\alpha\omega\eta)=h(\omega)$,
contradicting the choice of $\gamma$.
This proves $\gamma_{-1}>\delta_1$.

Now assume that $h(\omega)=1$. 
Then \cref{pf:omega} reduces to
\[
h(\alpha\omega\eta)=h(\omega)-2=-1,
\]
contradicting the premise $\pi\in\B$.
This completes the proof.
\end{proof}

As will be seen, \cref{thm:X=Y} is the key in the new proof 
of \cref{thm:b=p}.

\begin{theorem}\label{thm:X=Y}
Let $\lambda=in(j-1)j$ and $\mu=(j-1)jni$ where \hbox{$i+2\le j\le n-1$}.
The map 
$(\alpha,\lambda,\gamma,\delta)
\mapsto(\gamma', \mu, \alpha', \delta)$
is a bijection between the sets $X_n(\lambda)$ and $X_n(\mu)$, with the inverse map
$(\alpha,\mu,\gamma,\delta)
\mapsto(\gamma', \lambda, \alpha', \delta)$.
\end{theorem}
\begin{proof}
For $\pi=(\alpha,\lambda,\gamma,\delta)\in X_n(\lambda)$, define
$f(\pi)=\gamma'\mu\alpha'\delta$.
We shall show that 
$f(\pi)\in X_n(\mu)$ and 
the $\mu$-decomposition of $f(\pi)$
is $(\gamma', \mu, \alpha', \delta)$, i.e.,
\begin{enumerate}[label=(\arabic*)]
\item\label[itm]{itm:f:B}
$f(\pi)\in\B$;
\item\label[itm]{itm:h=h}
$h(\gamma'\mu\alpha')=h(\mu)$;
\item\label[itm]{itm:c}
$\alpha\mu_{-1}\in\B$; and
\item\label[itm]{itm:longest}
$\alpha'$ is the longest factor of $f(\pi)$ satisfying \cref{itm:c,itm:h=h}.
\end{enumerate}
Note that \cref{itm:f:B,itm:h=h,itm:c} are independent to each other. 
We show \cref{itm:h=h} first for it will be of use in showing \cref{itm:f:B}.

\noindent\cref{itm:h=h}
Since the integers $\lambda_{-1}=j$ and $\mu_1=j-1\notin\A(\gamma)$ are adjacent, 
we obtain
\[
h(\gamma'\mu_1)=h(\gamma'\lambda_{-1})=-h(\lambda_{-1}\gamma).
\]
Since $h(\lambda)=h(\mu)=1$,
we can verify that
\begin{align}
h(\gamma'\mu\alpha')
&=h(\gamma'\mu_1)
+h(\mu)
+h(\mu_{-1}\alpha')\notag\\
&=-h(\lambda_{-1}\gamma)
+(2-h(\lambda))
-h(\alpha\lambda_1)\notag\\
&=2-h(\alpha\lambda\gamma)
=2-h(\lambda)=1.\label{pf:h:gma}
\end{align}

\noindent\cref{itm:f:B}
Since $\pi=(\alpha,\lambda,\gamma,\delta)$, we have $\gamma' \lambda_{-1}\in\B$.
Replacing $\lambda_{-1}=j$ by the adjacent integer $j-1$ gives $\gamma'\mu_1\in\B$.
Since $\mu\in\B$, we find
$\gamma'\mu\in\B$ immediately.
Now, since $\alpha\in\B$,
for any factorization $\alpha'=\rho\tau$, we have
$h(\tau)=-h(\tau')\le 0$. Using \cref{pf:h:gma}, we can deduce that
\[
h(\gamma'\mu\rho)
\ge
h(\gamma'\mu\rho)+h(\tau)
\ge h(\gamma'\mu\alpha')-1
=0.
\]
Hence $\gamma'\mu\alpha'\in\B$.
Since $h(\lambda)=1$, \cref{lem:delta} implies
$\delta\in\B$.
Together with \cref{pf:h:gma}, we conclude $\gamma'\mu\alpha'\delta\in\B$. 

\noindent\cref{itm:c}
It is true since the word $\alpha\mu_{-1}=\alpha\lambda_1$ 
is a prefix of the ballot permutation~$\pi$, and thus a ballot one.

\noindent\cref{itm:longest}
Assume that $\alpha'$
is not the longest factor of $f(\pi)$ satisfying \cref{itm:c,itm:h=h}.
Then~$\delta$ has a nonempty prefix $\rho$ such that $\rho'\alpha\mu_{-1}\in\B$ and 
\[
h(\gamma'\mu\alpha'\rho)=1.
\]
Together with  \cref{pf:h:gma}, we find
\begin{equation}\label{pf:h:ar}
h(\alpha_{-1}'\rho)
=h(\gamma'\mu\alpha'\rho)-h(\gamma'\mu\alpha')=0.
\end{equation}
On the other hand, since $\rho'\alpha\mu_{-1}$ is ballot, so is its prefix $\rho'$.
Since $\delta$ is ballot, so is its prefix $\rho$.
Therefore, the word $\rho$ must be of height $0$,
and $h(\alpha_{-1}'\rho)\in\{\pm1\}$, contradicting \cref{pf:h:ar}.

It remains to show that $f\colon X_n(\lambda)\to X_n(\mu)$ is a bijection.
For $\sigma=(\alpha,\mu,\gamma,\delta)\in X_n(\mu)$, define
$g(\sigma)=\gamma'\lambda\alpha'\delta$.
Same to the above, one may show that $g(\sigma)\in X_n(\lambda)$ 
and the $\lambda$-decomposition of $g(\sigma)$ 
is $(\gamma', \lambda, \alpha', \delta)$.
By definition, we can derive that
\[
gf(\pi)
=g\brk1{(\gamma', \mu, \alpha', \delta)}
=\brk1{(\alpha')',\lambda,(\gamma')',\delta}
=\pi.
\]
Thus the composition $gf$ is the identity on $X_n(\lambda)$.
In the same fashion, one may show that $fg$ is the identity on $X_n(\mu)$.
Hence $f$ is a bijection, with the inverse~$g$.
\end{proof}

Now we are in a position to give the new proof of \cref{thm:b=p}.

\begin{proof}[Proof of \cref{thm:b=p}]
Let $n\ge 4$, $i\ge 1$ and $i+2\le j\le n-1$.
Let $b_n(i,j)$ be the number of ballot permutations in $\B_n$ 
containing the factor $inj$.
Recall that $\lambda=in(j-1)j$ and $\mu=(j-1)jni$.
We claim that
\begin{equation}\label{X:lambda}
\abs{X_n(\lambda)}=b_n(i,\,j-1)-b_n(i,j).
\end{equation}

In fact, consider the involution $\Phi$ of 
exchanging the letters $j-1$ and $j$ on the set
\[
Y_n=\B_n(i,\,j-1)\setminus X_n(\lambda).
\]
We shall show that $\Phi(Y_n)=\B_n(i,j)$, which 
implies \cref{X:lambda} immediately. Let $\pi\in Y_n$.
If the letters $j-1$ and $j$ are not adjacent in $\pi$, 
then $\Phi(\pi)\in\B_n(i,j)$ and
the letters $j-1$ and $j$ are not adjacent in $\Phi(\pi)$.
Suppose that $j-1$ and $j$ are adjacent in $\pi$.
Since $\pi\not\in X_n(\lambda)$,
there is no factor~$\gamma$ satisfying~\cref{cond:omega}.
In other words, the height of any prefix of $\pi$ that 
is longer than $\alpha\lambda$ is at least 2, where $\alpha\lambda$ 
is the prefix of $\pi$ ending at $\lambda$.
Therefore $\Phi(\pi)\in \B_n(i,j)$.
It is clear that the preimage of every permutation $\sigma\in\B_n(i,j)$
lies in $Y_n$.
This proves \cref{X:lambda}.

Similarly, one may show that
\begin{equation}\label{X:mu}
\abs{X_n(\mu)}=b_n(j,i)-b_n(j-1,\,i).
\end{equation}
By \cref{thm:X=Y,X:lambda,X:mu,lem:|i-j|=1}, we infer that
\begin{align*}
b_n(i,j)+b_n(j,i)
&=b_n(i,\,j-1)+b_n(j-1,\,i)\\
&=\dotsb=b_n(i,\,i+1)+b_n(i+1,\,i)=2b_{n-2}.
\end{align*}
Since the number of ballot permutations in $\B_n$
ending with the letter $n$ is~$b_{n-1}$, we obtain
\[
b_n=b_{n-1}+\sum_{i,\,j}b_n(i,j)
=b_{n-1}+(n-1)(n-2)b_{n-2}.
\]
It is trivial to check that $b_1=b_2=1$.
Since the sequence $p_n$ admits the same recurrence and initial values, 
we conclude that $b_n=p_n$. 
\end{proof}

\section{A Toeplitz property}\label{sec:Toeplitz}
Computer calculus gives that the matrices $\brk1{b_n(i,j)}_{i,j}$
for $3\le n\le 8$ are respectively
\[
\left[ \begin {array}{cc} 0&1\\ \noalign{\medskip}1&0\end {array}
\right],\quad
 \left[ \begin {array}{ccc} 0&1&0\\ \noalign{\medskip}1&0&1
\\ \noalign{\medskip}2&1&0\end {array} \right],\quad
 \left[ \begin {array}{cccc} 0&3&2&1\\ \noalign{\medskip}3&0&3&2
\\ \noalign{\medskip}4&3&0&3\\ \noalign{\medskip}5&4&3&0\end {array}
 \right],\quad
\left[ \begin {array}{ccccc} 0&9&6&3&1\\ \noalign{\medskip}9&0&9&6&3
\\ \noalign{\medskip}12&9&0&9&6\\ \noalign{\medskip}15&12&9&0&9
\\ \noalign{\medskip}17&15&12&9&0\end {array} \right],
\]
\[
\left[ \begin {array}{cccccc} 0&45&36&27&19&13\\ \noalign{\medskip}45
&0&45&36&27&19\\ \noalign{\medskip}54&45&0&45&36&27
\\ \noalign{\medskip}63&54&45&0&45&36\\ \noalign{\medskip}71&63&54&45&0
&45\\ \noalign{\medskip}77&71&63&54&45&0\end {array} \right],
\quad\text{and}\quad
\left[ \begin {array}{ccccccc} 0&225&182&139&99&65&38
\\ \noalign{\medskip}225&0&225&182&139&99&65\\ \noalign{\medskip}268&
225&0&225&182&139&99\\ \noalign{\medskip}311&268&225&0&225&182&139
\\ \noalign{\medskip}351&311&268&225&0&225&182\\ \noalign{\medskip}385
&351&311&268&225&0&225\\ \noalign{\medskip}412&385&351&311&268&225&0
\end {array} \right].
\]
A square matrix $(a_{ij})$ is said to be \emph{Toeplitz} if 
$a_{i+1,\,j+1}=a_{i,j}$ for all well defined entries $a_{i+1,\,j+1}$ and~$a_{i,j}$.
In this section, we manage to show that the matrices 
\[
B(n,d)=\brk1{b_{n,d}(i,j)}_{i,j=1}^{n-1}
\quad\text{and}\quad
P(n,d)=\brk1{p_{n,d}(i,j)}_{i,j=1}^{n-1}
\]
are Toeplitz for any number $d$.

\begin{theorem}\label{thm:Toeplitz:B}
The matrix $B(n,d)$ is Toeplitz for all $n$ and $d$.
\end{theorem}

\begin{proof}
Suppose that $n\ge 4$, $1\le d\le \floor{(n-1)/2}$, and $i, j\in[n-2]$. 
This proof is organized as follows. First we give necessary notion and notation 
to define a map~$T$ between the sets $\B_{n,d}(i,j)$ and $\B_{n,d}(i+1,\,j+1)$.
Second we interpret some facts implicitly contained in the definition.
We then make some efforts to show that~$T$ is well defined. 
Finally we prove that $T$ is bijective, which will complete the proof.

Let $m=\min(i,j)$ and $M=\max(i,j)$.
We introduce the notation
\[
\overline{x}=\begin{cases}
M+1,&\text{if $x=M$},\\[3pt]
x,&\text{otherwise},
\end{cases}
\quad\text{and}\quad
\underline{x}=\begin{cases}
m,&\text{if $x=m+1$},\\[3pt]
x,&\text{otherwise},
\end{cases}
\]
and define the words
\[
\overline{(m;l)}=\overline{m(m+1)\dotsm(m+l-1)}
\quad\text{and}\quad
\underline{(M+1;\,l)}=
\underline{(M+1)M\dotsm(M-l+2)}.
\]
In particular, $\overline{(m;0)}=\underline{(M+1;\,0)}=\epsilon$.
For $\pi\in \B_{n,d}(i,j)$, 
define the \emph{lower width}~$l_\pi$ to be the length~$l$ of 
the longest factor of the form $\overline{(m;l)}$ 
or $\overline{(m;l)}'$, if the letters~$M$ and $M+1$ are not adjacent in $\pi$; 
and define $l_\pi=0$ if $M$ and $M+1$ are adjacent.
Note that $l_\pi\le M-m+1$
since the letter $M+1$ occurs twice in the word $\overline{(m;\,M-m+2)}$.
Define the \emph{lower core} of $\pi$ to be the word
\begin{equation}\label{def:kappa:pi}
\kappa_\pi=\begin{cases}
\overline{(i;\,l_\pi)}'nj,&\text{if $i<j$},\\[3pt]
in\overline{(j;\,l_\pi)},&\text{if $i>j$}.
\end{cases}
\end{equation}
Then $\kappa_\pi$ is a factor of $\pi$ of length $l_\pi+2$.
For example, 
\[
\kappa_{382549671}=96
\quad\text{and}\quad		
\kappa_{134875962}=7596.
\]
Define a map 
\[
T\colon \B_{n,d}(i,j)\to\B_{n,d}(i+1,\,j+1)
\]
by firstly replacing the lower core $\kappa_\pi$ by the word
\begin{equation}\label{def:T:kappa:pi}
T|_{\kappa_\pi}=\begin{cases}
(i+1)n\underline{(j+1;\,l_\pi)},&\text{if $i<j$},\\[3pt]
\underline{(i+1;\,l_\pi)}'n(j+1),&\text{if $i>j$},
\end{cases}
\end{equation}
and secondly substituting the letter~$\overline{x}$ 
by the letter $\underline{x-l_\pi+1}$ 
for each element $x\in[m+l_\pi,\,M]\backslash\A(\kappa_\pi)$. 
For example, 
\[
T(382549671)=382645971
\quad\text{and}\quad
T(134875962)=134869752.
\]

We call the first operation the \emph{core replacement}, and the second
the \emph{straightening}.
We will show that $T$ is well defined, i.e., 
$T(\pi)\in \B_{n,d}(i+1,\,j+1)$ for any $\pi\in\B_{n,d}(i,j)$, 
and that $T$ is bijective.
From definition, we observe that the map $T$ operates in the following 3 steps:
\begin{enumerate}
\item
The core replacement is length-preserving and position-preserving, 
namely the words $T|_{\kappa_\pi}$ and $\kappa_\pi$ have the same length 
and the same position.
Every letter in the set 
\[
\A(\kappa_\pi)=\{\ovU{m},\,\ovU{m+1},\,\dots,\,\ovD{m+l_\pi-1},\,M,\,n\}
\]
is replaced by a letter in the set 
\begin{equation}\label{pf:A:T:kappa}
\A\brk1{T|_{\kappa_\pi}}
=\{m+1,\,\underline{M-l_\pi+2},\,\underline{M-l_\pi+3},
\,\dots,\,\underline{M+1},\,n\}.
\end{equation}
In particular, the letter $n$ is contained in both words 
$\kappa_\pi$ and $T|_{\kappa_\pi}$. 
\item
The straightening operation maps the letters in the set 
\[
[m,\,M+1]\backslash\A\brk1{\kappa_\pi}
=\brk[c]1{\ovD{m+l_\pi},\,\ovD{m+l_\pi+1},\,\dots,\,\overline{M}}
\]
to the letters in the set 
\[
[m,\,M+1]\backslash\A\brk1{T|_{\kappa_\pi}}
=\brk[c]1{\underline{m+1},\,\underline{m+2},\,\dots,\,\underline{M-l_\pi+1}}
\]
in the order-preserving manner, where $\overline{M}=M+1$ and $\underline{m+1}=m$.
\item
The letters that is not dealt with in the first two steps constitute
the union $[m-1]\cup[M+2,\,n-1]$, and are unchanged by $T$.
\end{enumerate}
From the above observations we can infer that $T(\pi)\in\mathfrak{S}_n$.
Fix $\pi\in \B_{n,d}(i,j)$. 
Then~$\pi$ induces a unique bijection 
$T_\pi\colon[n]\to[n]$ which maps the letter~$\pi_i$
to the letter $T(\pi)_i$ for all $i\in[n]$.
Note that the preimage of every letter in the interval $[m,\,M+1]$ under~$T_\pi$
must be in the same set $[m,\,M+1]$.

Now, we show that $T(\pi)\in \B_{n,d}(i+1,\,j+1)$. 
More precisely, we need to show 
\begin{enumerate}[label=(\roman*)]
\item\label[itm]{itm:pf:i+1:j+1}
the word $T(\pi)$ contains the factor $(i+1)n(j+1)$;
\item\label[itm]{itm:pf:des=des}
$\des\brk1{T(\pi)}=\des(\pi)$; and
\item\label[itm]{itm:pf:T:B}
$T(\pi)\in\B$.
\end{enumerate}

Suppose that $l_\pi=0$.
Then the $\kappa_\pi\in\{nM,\,Mn\}$,
and $\pi$ contains the factor 
\[
\iota=mnM(M+1)
\]
or its reversal. 
From definition of the core replacement,
$T|_{\kappa_\pi}$ is the factor $(m+1)n$ or its reversal.
Since the letter $m$ (resp., $M+1$) is the smallest (resp., largest)
one in the interval $[m,\,M+1]$,
it is invariant in the order-preserving straightening.
Therefore, the image $T(\pi)$ contains the factor 
\[
T|_\iota=m(m+1)n(M+1)
\]
or its reversal. This proves \cref{itm:pf:i+1:j+1}.
Note that each of the words 
$\iota$ and $T|_\iota$
has a unique descent, 
and the reversal operation exchanges descents and ascents.
Since the straightening is order-preserving,
the map $T$ preserves all descents 
and ascents that has empty intersection with the cores. As a result, 
we obtain \cref{itm:pf:des=des}. Moreover,
the unique descent in $T|_\iota$ appears later 
than the unique descent appears in $\iota$, 
and the unique ascent in $T|_\iota'$ appears earlier than that in $\iota$.
Since $\pi$ is ballot, so is~$T(\pi)$.
This proves \cref{itm:pf:T:B} and $T(\pi)\in \B_{n,d}(i+1,\,j+1)$ for $l_\pi=0$.

Below we can suppose that $l_\pi\ge 1$. 
Then \cref{itm:pf:i+1:j+1} is clear from the definition of the core replacement.
In order to show \cref{itm:pf:des=des}, 
we claim the equivalence
\begin{equation}\label[eqrl]{pf:eqrl}
u<v\iff T_\pi(u)<T_\pi(v)
\end{equation}
for any adjacent letters $u$ and $v$ such that $\{u,v\}\not\subseteq\A(\kappa_\pi)$.
Since the straightening is order-preserving,
\cref{pf:eqrl} is true if $\{u,v\}\cap\A(\kappa_\pi)=\emptyset$. 
Below we can suppose that 
\[
v\in\{M,\,\ovD{m+l_\pi-1}\}\subseteq\A(\kappa_\pi)
\quad\text{and}\quad
u\not\in\A(\kappa_\pi).
\]

\smallskip

If $v=\ovD{m+l_\pi-1}$,
then $T_\pi(v)=m+1$ by the core replacement.
Since $u\not\in\A(\kappa_\pi)$, the letter $u$ is not mapped in the core replacement.
If $T_\pi(u)=m$, then the letter~$u$ is mapped in the straightening operation.
Since $m$ is the smallest image in the straightening, 
its preimage $u$ must be the smallest element in the straightening, 
that is, $u=T_\pi^{-1}(m)=\ovD{m+l_\pi}$.
Now, the elements $\ovD{m+l_\pi}=u$ and $\ovD{m+l_\pi-1}=v$ are adjacent,
contradicting the definition of the lower width $l_\pi$. Therefore,  $T_\pi(u)\ne m$.
We then obtain the equivalence
\[
T_\pi(u)<T_\pi(v)\iff
T_\pi(u)\le m-1.
\]
On the other hand, 
by the premise $u\not\in\A(\kappa_\pi)$ and $v\in\A(\kappa_\pi)$,
we deduce the equivalence
\[
u<v\iff
u\le \overline{m}-1=m-1.
\]
Since all letters in the set $[m-1]$ are fixed under $T_\pi$, 
we obtain the desired \cref{pf:eqrl} by combining 
the above two equivalences.

\smallskip

If $v=M$, then $T_\pi(v)=\underline{M-l_\pi+2}$.
Since $l_\pi\ne0$,
we find $u\ne M+1$. Since $u\not\in\A(\kappa_\pi)$, we obtain the equivalence
\[
u>v\iff 
u\in[M+2,\,n-1].
\]
On the other hand, since $u\not\in\A(\kappa_\pi)$, we deduce 
$T_\pi(u)\not\in\A\brk1{T(\kappa_\pi)}$.
In view of \cref{pf:A:T:kappa}, we obtain the equivalence
\[
T_\pi(u)>T_\pi(v)\iff
T_\pi(u)\in[M+2,\,n-1].
\]
Since the map $T_\pi$ restricted on the set $[M+2,\,n-1]$ is the identity, 
we obtain \cref{pf:eqrl} as desired.

\smallskip

By \cref{pf:eqrl}, the map $T$ preserves all descents and ascents
that are not entirely contributed by the lower core. 
For those entirely contained in the lower core, 
we note both the numbers of descents of the words $\kappa_\pi$ 
and~$T|_{\kappa_\pi}$ equal the lower width $l_\pi$.
Therefore, 
$T$ preserves the total number of descents.
This proves \cref{itm:pf:des=des}.

\smallskip

For \cref{itm:pf:T:B}, we observe that in the core replacement, either 
\begin{itemize}
\item
both the preimage $\kappa_\pi$ and its image $T|_{\kappa_\pi}$
contain a unique ascent,
and the ascent in $T|_{\kappa_\pi}$ 
appears earlier than the ascent in~$\kappa_\pi$ appears, or
\item
both of them contain a unique descent,
and the decent in $T|_{\kappa_\pi}$ 
appears later than the descent in $\kappa_\pi$ appears.
\end{itemize}
Since $T$ preserves all other descents and ascents,
the premise $\pi\in\B$ implies \cref{itm:pf:T:B}.
This confirms that $T(\pi)\in \B_{n,d}(i+1,\,j+1)$.

\smallskip

It remains to show that $T$ is bijective.
We define the \emph{upper width} $l_\sigma'$ to be the length $l$
of the longest factor of the form $\underline{(M+1;\,l)}$ 
or $\underline{(M+1;\,l)}'$,
if the letters $m$ and $m+1$ are not adjacent in $\sigma$; 
and define $l_\sigma'=0$ otherwise.
Define the \emph{upper core}
of $\sigma$ to be the word
\[
\rho_\sigma=\begin{cases}
(i+1)n\underline{(j+1;\,l_\sigma')},&\text{if $i<j$},\\[3pt]
\underline{(i+1;\,l_\sigma')}'n(j+1),&\text{if $i>j$}.
\end{cases}
\]
Define a map 
\[
T'\colon \B_{n,d}(i+1,\,j+1)\to B_{n,d}(i,j)
\]
as follows.
For $\sigma\in \B_{n,d}(i+1,\,j+1)$, 
define $T'(\sigma)$ to be the permutation obtained 
from $\sigma$ by firstly replacing its upper core by
\[
T'|_{\rho_\sigma}
=\begin{cases}
\overline{(i;\,l_\sigma')}'nj,&\text{if $i<j$},\\[3pt]
in\overline{(j;\,l_\sigma')},&\text{if $i>j$},
\end{cases}
\]
and then replacing each letter in the set $[m,\,M+1]\backslash\A(\rho_\sigma)$
to a letter in the set $[m,\,M+1]\backslash\A(T'|_{\rho_\sigma})$
in the order-preserving manner. 

We shall show that
\begin{equation}\label{dsr:l=l}
l_\pi=l_{T(\pi)}'.
\end{equation}

In fact, if $l_\pi=0$, then $\pi$ contains the factor
$\iota$ or its reversal. 
It follows that $\sigma$ contains the factor~$T|_\iota$ or its reversal. 
Thus $l_\sigma'=0=l_\pi$.
Suppose that $l_\pi\ge 1$. 
Then $m+1$ and $m$ are not adjacent in $\sigma$, since otherwise
the letters 
\[
\ovD{m+l_\pi-1}=T_\pi^{-1}(m+1)
\quad\text{and}\quad
\ovD{m+l_\pi}=T_\pi^{-1}(m)
\]
would be adjacent in $\pi$, which
contradicts the definition of $l_\pi$. 
Thus $l_\sigma'\ge l_\pi$ from the definition of the core replacement.
Furthermore, the letters $\underline{M-l_\pi+2}$ and $\underline{M-l_\pi+1}$
are not adjacent in $\sigma$, since otherwise the letters 
\[
M=T_\pi^{-1}\brk1{\underline{M-l_\pi+2}}
\quad\text{and}\quad
M+1=T_\pi^{-1}\brk1{\underline{M-l_\pi+1}}
\]
would be adjacent in $\pi$, which contradicts $l_\pi\ge 1$.
This proves \cref{dsr:l=l}.

In the same fashion one may prove that $l_\sigma'=l_{T'(\sigma)}$.
As a consequence, both compositions $TT'$ and~$T'T$ are identities. 
This completes the whole proof.
\end{proof}

We remark that the map $T_\pi$ 
can be represented letter-wise using a piecewise function,
with the aid of the number $l_\pi$.

Note that the key notion of lower and upper core widths, 
and that of lower and upper cores,
are defined in a local structure of a permutation.
In fact, the map~$T$ looks for the 
longest factor consisting of discretely continuous numbers 
starting from $m$.
This localness perspective inspires us 
to translate the idea application from permutations to certain cyclic words,
especially to each cycle of a permutation in its cycle representation, 
and obtain \cref{thm:Toeplitz:P}.

\begin{theorem}\label{thm:Toeplitz:P}
The matrix $P(n,d)$ is symmetric and Toeplitz for all $n$ and $d$.
\end{theorem}
\begin{proof}
In this proof, we keep in mind that permutations are considered to be
unions of cycles, and forget their nature as maps on the set $[n]$.
Since taking reversal cycle-wise 
is an involution between the sets $\P_{n,d}(i,j)$ and $\P_{n,d}(j,i)$,
the matrix $P(n,d)$ is symmetric.
Let $n\ge 4$ and $1\le d\le \floor{(n-1)/2}$. 
In order to show the Toeplitz property, namely $P(n,d)_{i,j}=P(n,d)_{i+1,\,j+1}$,
it suffices to consider $1\le i<j\le n-2$ by virtue of the symmetry.

For $\pi\in \P_{n,d}(i,j)$,
let $C_\pi$ be the cycle of $\pi$ containing the maximum letter $n$. 
Paralleling the notion of lower width by considering 
cyclic factors instead of factors, we define the lower width $l_\pi$ 
to be the length $l$ of the longest \emph{cyclic factor} of $C_\pi$ that is
of the form $\overline{(m;l)}$ or $\overline{(m;l)}'$,
if the letters $M$ and $M+1$ are not adjacent in~$C_\pi$; 
and define $l_\pi=0$ otherwise.
The subsequent notion of the lower core and core replacement 
are exactly the same as those in the proof of \cref{thm:Toeplitz:B},
using the cyclic version of the lower width.
Define the map 
\[
T\colon\P_{n,d}(i,j)\to\P_{n,d}(i+1,\,j+1)
\]
by running firstly the core replacement and then the straightening.

While the core replacement affects letters only inside the cycle $C_\pi$,
the straightening works on letters that are both inside and outside~$C_\pi$.
Same to the proof of \cref{thm:Toeplitz:B}, we can deduce the following one by one:
\begin{enumerate}[label=(\roman*)]
\item
the width $l_\pi\le M-m+1$, where $M=\max(i,j)$ and $m=\min(i,j)$;
\item
the image $T(\pi)$ is a permutation of length~$n$, 
and contains the cyclic factor $(i+1)n(j+1)$;
\item
the map $T$ preserves the length of each cycle of~$\pi$, and
the permutation~$T(\pi)$ is of odd order;
\item
the map $T$ preserves the cyclic descents outside $\kappa_\pi$,
and thus preserves the cyclic descent number of each cycle except $C_\pi$;
\item
the map $T$ preserves the cyclic descent number inside the core $\kappa_\pi$,
and the cyclic descents formed by one letter inside $\kappa_\pi$
and the other letter outside~$\kappa_\pi$; 
thus $T$ preserves the cyclic descent number of $C_\pi$ if $C_\pi\ne(\kappa_\pi)$;
\item\label[itm]{itm:pf:C=kappa}
the map $T$ preserves the cyclic descent number of $C_\pi$ if $C_\pi=(\kappa_\pi)$.
\end{enumerate}
In conclusion, the map $T$ preserves the cyclic descent number of each cycle,
thus preserves the cyclic weight $d$ of $\pi$.

We shall show \cref{itm:pf:C=kappa} in detail.
Suppose that $C_\pi=(\kappa_\pi)$. We claim that $1\le l_\pi\le M-m$.
In fact, if $l_\pi=0$, then $\abs{C_\pi}=\abs{\kappa_\pi}=2$, 
contradicting that $\abs{C_\pi}$ is odd. If $l_\pi=M-m+1$, 
then the letter $M+1=\underline{m+l_\pi}$ is adjacent to the letter $M$,
and $l_\pi=0$, a contradiction. This proves the claim.
Now the image $T(\pi)$ contains the cycle $\brk1{T|_{\kappa_\pi}}$.
As a consequence, one finds
\begin{itemize}
\item
if $i<j$, then $\casc(\kappa_\pi)=1=\casc(T|_{\kappa_\pi})$;
\item
if $i>j$, then $\cdes(\kappa_\pi)=1=\cdes(T|_{\kappa_\pi})$.
\end{itemize}
This proves \cref{itm:pf:C=kappa} and completes the proof.
\end{proof}

For example, if
$\pi=(1,6,8,2,10)(3,12,9,11,7,5,4)\in\P_{12,4}(3,9)$,
then 
\[
C_\pi=(3,12,9,11,7,5,4),\ 
l_\pi=3,\ 
\kappa_\pi=(5,4,3,12,9)\ \text{and}\ 
T|_{\kappa_\pi}=(4,12,10,9,8).
\]
The straightening maps the letters in 
\[
[3,10]\backslash\A(\kappa_\pi)=\{6,7,8,10\}
\]
to the letters in 
\[
[3,10]\backslash\A\brk1{T|_{\kappa_\pi}}=\{3,5,6,7\}
\]
in the order-preserving manner. As a result, we obtain
\[
T(\pi)=(1,3,6,2,7)(10,9,8,11,5,4,12)\in\P_{12,4}(4,10).
\]

We propose the following refinement conjecture of \cref{conj:Spiro}.

\begin{conjecture}\label{conj:nd1j}
$b_{n,d}(1,j)+b_{n,d}(j,1)=2p_{n,d}(1,j)$.
\end{conjecture}

\begin{theorem}
\cref{conj:Spiro} is true if \cref{conj:nd1j} is true.
\end{theorem}
\begin{proof}
Fix $n,d,i,j$.
Since $b_{n,d}(1,j)+b_{n,d}(j,1)=2p_{n,d}(1,j)$ for all $n$, $d$ and $j$, 
we can infer by \cref{thm:Toeplitz:P,thm:Toeplitz:B} that
\[
b_{n,d}(i,j)+b_{n,d}(j,i)=2p_{n,d}(i,j).
\]
Since the number of permutations in $\B_{n,d}$ with 
the letter $n$ appearing at the last position is $b_{n-1,\,d}$,
we obtain
\begin{equation}\label{eq:bnd}
b_{n,d}=b_{n-1,\,d}+\sum_{i\ne j}b_{n,d}(i,j).
\end{equation}
Since the number of permutations in $\P_{n,d}$ with the letter~$n$ 
forming a singleton is~$p_{n-1,\,d}$, 
we obtain
\begin{equation}\label{eq:pnd}
p_{n,d}=p_{n-1,\,d}+\sum_{i\ne j}p_{n,d}(i,j).
\end{equation}
Since $p_{n,d}(i,j)=p_{n,d}(j,i)$ and \cref{conj:Spiro} holds 
for $d=\floor{(n-1)/2}$, by induction on $n$, we derive \cref{conj:Spiro}
by the recurrences \cref{eq:bnd,eq:pnd}.
\end{proof}

\section{Easy cases of \cref{conj:nd1j}}\label{sec:easy}

In this section we solve some easy cases for \cref{conj:nd1j}. 

\begin{proposition}
\Cref{conj:nd1j} is true for $d=0$ and $d=1$.
\end{proposition}
\begin{proof}
When $d=0$, both sides the desired equality equal to $1$. Let $d=1$.
By \cref{lem:|i-j|=1}, we can infer that
\[
b_{n,1}(1,2)=b_{n,1}(2,1)=b_{n-2,\,0}=1
\quad\text{and}\quad 
p_{n,1}(1,2)=p_{n-2,\,0}=1.
\]
For $j\ge3$, it is easy to show that 
\[
b_{n,1}(j,1)=2^{j-2},\qquad
b_{n,1}(1,j)=0,\quad\text{and}\quad
p_{n,1}(1,j)=2^{j-3}.
\]
This verifies \cref{conj:nd1j} for $d=1$ by direct calculation.
\end{proof}

\begin{lemma}\label{lem:pnd:12=13}
$p_{n,d}(1,2)=p_{n,d}(1,3)$.
\end{lemma}
\begin{proof}
We define a map $\phi\colon \P_{n,d}(1,2)\to\P_{n,d}(1,3)$ as follows. 
Let $\pi\in\P_{n,d}(1,2)$.
By considering the exchange of the letters $2$ and $3$,
we can suppose that there exists a nonempty word $\alpha$ such that~$\pi$
contains a cycle $c=(1n23\alpha)$. 
Define $\phi(\pi)$ to be the permutation obtained from $\pi$ 
by replacing the cycle $c$ by the cycle $u=(1n32\alpha')$.
It is clear that $\phi$ is an involution.
Let $d=\des(\alpha)$.
Then 
\[
\cdes(c)=d+2=\casc(u)
\quad\text{and}\quad
\casc(c)=\ell(c)-d-2=\cdes(u).
\]
It follows that $c$ and $u$ have the same weight, 
and so do the permutations $\pi$ and~$\phi(\pi)$. 
This completes the proof.
\end{proof}

\begin{proposition}
If \cref{conj:nd1j} is true for $j=2$,
then it is true for $j=3$.
\end{proposition}
\begin{proof}
Suppose that \cref{conj:nd1j} is true for $j=2$, that is,
\[
b_{n,d}(1,2)+b_{n,d}(2,1)=2p_{n,d}(1,2).
\]
For the sake of showing \cref{conj:nd1j} for $j=3$, i.e.,
\[
b_{n,d}(1,3)+b_{n,d}(3,1)=2p_{n,d}(1,3),
\]
by \cref{lem:pnd:12=13}, it suffices to show that
\[
b_{n,d}(1,2)-b_{n,d}(1,3)
=
b_{n,d}(3,1)-b_{n,d}(2,1).
\]
For words $u$ and $v$, let $b_{n,d}(u,v)$ 
be the number of permutations in $\B_{n,d}$ that contain the factor $unv$.
By exchanging the letters $2$ and $3$, we obtain
\begin{align}
b_{n,d}(1,2)-b_{n,d}(1,3)
&=b_{n,d}(1,23)-b_{n,d}(1,32),\quad\text{and}\quad\label{pf:1}\\
b_{n,d}(3,1)-b_{n,d}(2,1)
&=b_{n,d}(23,1)-b_{n,d}(32,1).\label{pf:2}
\end{align}
Using the proof of \cref{lem:|i-j|=1}, one may show that 
\begin{align}
b_{n,d}(1,23)&=b_{n-3,\,d-1}=b_{n,d}(23,1)\quad\text{and}\quad\label{pf:3}\\
b_{n,d}(1,32)&=b_{n-3,\,d-2}=b_{n,d}(32,1).\label{pf:4}
\end{align}
Combining \cref{pf:1,pf:2,pf:3,pf:4}, we can deduce 
\begin{align*}
b_{n,d}(1,2)-b_{n,d}(1,3)
&=b_{n,d}(1,23)-b_{n,d}(1,32)\\
&=b_{n,d}(23,1)-b_{n,d}(32,1)
=b_{n,d}(3,1)-b_{n,d}(2,1).
\end{align*}
This completes the proof.
\end{proof}

\section*{Acknowledgment}
We are grateful to Professor Richard Stanley for his encouragement 
of working on this project, 
and to Professor Olivier Bernardi for his interests of the new proof
of \cref{thm:b=p}. 
The main part of this paper was completed when the first author 
was a visiting scholar at MIT.

\end{document}